\documentclass[article, 12pt]{amsart}

\usepackage{amssymb,amsmath,amsfonts,amsthm}
\usepackage{stmaryrd}
\usepackage[all,cmtip,poly]{xy}
\usepackage{svn}
\usepackage{enumerate}
\usepackage{mathrsfs}
\usepackage[latin1]{inputenc}
\usepackage{verbatim}   % This is for the begin{comment} environment.
\usepackage{amscd}

\theoremstyle{plain}
\newtheorem{thm}{Theorem}[section] %here!! Section or subsection, about the counter!!!

\newtheorem{defn}[thm]{Definition}
\newtheorem{theorem}[thm]{Theorem}
\newtheorem{question}[thm]{Question}
\newtheorem{prop}[thm]{Proposition}

\newtheorem{lemma}[thm]{Lemma}

\newtheorem{cor}[thm]{Corollary}

\theoremstyle{definition}

\def \gs {\mathfrak S}
\def \sfi {\mathrm{Mod}_{\gs}^{\varphi, r}}

\def \hR {{\widehat{\mathcal R} }}
\def \hM {{\hat \M}}
\def \O {\mathcal O}
\def \E {\mathcal E}
\def \t {\mathrm}
\def \FrR {\mathrm{Fr} R}
\def \inj {\hookrightarrow }

\newcommand{\M}{\mathfrak{M}}

\newcommand{\Zp}{\mathbb{Z}_p}
\newcommand{\Qp}{\mathbb{Q}_p}

\newcommand{\Z}{\mathbb{Z}}

\newcommand{\N}{\mathfrak{N}}
\newcommand{\mfc}{\mathfrak c}
\newcommand{\mfS}{\mathfrak S}

\def \ito {\overset  \sim  \to}
\def \onto {\twoheadrightarrow}
\def \cT {\mathcal T}
\def \upi {\underline \pi }

\DeclareMathOperator{\Fr}{Fr}
\DeclareMathOperator{\Gal}{Gal}
\DeclareMathOperator{\gal}{Gal}

\DeclareMathOperator{\Ker}{Ker}
\DeclareMathOperator{\Mat}{Mat}

\DeclareMathOperator{\Rep}{Rep}

\newcommand{\cris}{\mathrm{cris}}

\newcommand{\st}{\mathrm{st}}

\newcommand{\ur}{\mathrm{ur}}
\newcommand{\tor}{\mathrm{tor}}
\newcommand{\fr}{\mathrm{fr}}

\newcommand{\huaS}{\mathfrak{S}}
\newcommand{\huaM}{\mathfrak{M}}

\newcommand{\huat}{\mathfrak{t}}
\newcommand{\mhat}{\hat{\huaM}}

\newcommand{\wt}{\widetilde}

\title{Limit of torsion semi-stable Galois representations with unbounded weights}
\author{HUI GAO}
\address{Department of Mathematics and Statistics, University of Helsinki, FI-00014, Finland}
\email{hui.gao@helsinki.fi}

\begin{document}

\begin{abstract}
Let $K$ be a complete discrete valuation field of characteristic $(0, p)$ with perfect residue field, and let $T$
be an integral $\Zp$-representation of $\Gal(\overline{K}/K)$. A theorem of T. Liu says that if $T/p^n T$ is torsion semi-stable (resp. crystalline) of \emph{uniformly bounded} Hodge-Tate weights for all $n \geq 1$, then $T$ is also semi-stable (resp. crystalline). In this note, we show that we can relax the condition of ``uniformly bounded Hodge-Tate weights" to an \emph{unbounded} (log-)growth condition.
\end{abstract}

\maketitle
\pagestyle{myheadings}
\markright{Limit of torsion semi-stable Galois representations}

\tableofcontents

\section{Introduction}
We first introduce some notations.
Let $p$ be a prime, $k$ a perfect field of characteristic $p$, $W(k)$ the ring of Witt vectors, $K_0 = W(k)[\frac{1}{p}]$ the fraction field, $K$ a finite totally ramified extension of $K_0$, $e =e(K/K_0)$ the ramification index and $G_K =\Gal(\overline{K}/K)$ the absolute Galois group for a fixed algebraic closure $\overline{K}$ of $K$.

We use $\Rep_{\Zp}^{\tor}(G_K)$ (resp. $\Rep_{\Zp}^{\fr}(G_K)$) to denote the category of finite $p$-power torsion (resp.  $\Zp$-finite free) representations of $G_K$.
Let $r$ be an integer in the range $[0, \infty]$ (including infinity).
We use $\Rep_{\Zp}^{\fr, \st, [-r, 0]}(G_K)$ (resp. $\Rep_{\Zp}^{\fr, \cris, [-r, 0]}(G_K)$) to denote the category of finite free $\Zp$-lattices in semi-stable (resp. crystalline) representations of $G_K$ with Hodge-Tate weights in the range $[-r, 0]$.

\begin{defn}
Let $r$ be an integer in the range $[0, \infty]$ (including infinity).
$T_\infty \in \Rep_{\Zp}^{\tor}(G_K)$ is called torsion semi-stable (resp. crystalline) of weight $r$ if there exist two objects $L$ and $L'$ in $\Rep_{\Zp}^{\fr, \st, [-r, 0]}(G_K)$ (resp. $\Rep_{\Zp}^{\fr, \cris, [-r, 0]}(G_K)$) such that $T_\infty =  L/L'$.
%We denote the category of all such $T$ as $\Rep_{\Zp}^{\tor, \st, [-r, 0]}(G_K)$ (resp. $\Rep_{\Zp}^{\tor, \cris, [-r, 0]}(G_K)$).
\end{defn}

The following result was first conjectured by Fontaine (\cite{Fon97}), and was fully proved in \cite{Liu07} (some partial results were known by work of Ramakrishna, Berger and Breuil, see \cite[\S 1]{Liu07} for a historical account).

\begin{thm}[\cite{Liu07}] \label{thm: Liu07}
Let $T \in \Rep_{\Zp}^{\fr}(G_K)$. Suppose that there exists an $r \in [0, \infty)$, such that $T/p^nT$ is torsion semi-stable (resp. crystalline) of weight $r$ for all $n \ge 1$, then $T\otimes_{\Zp}\Qp$ is semi-stable (resp. crystalline) with Hodge-Tate weights in $[-r, 0]$.
\end{thm}

It is \emph{necessary} to have $r < \infty$ in the above theorem, because of the following result.
\begin{thm}[{\cite[Thm. 3.3.2]{GL-OC}}]
Suppose $K$ is a finite extension of $\Qp$. For any $T_\infty \in \Rep_{\Zp}^{\tor}(G_K)$, it is torsion semi-stable (in fact, torsion crystalline).
\end{thm}

In fact, suppose $T$ is in $\Rep_{\Zp}^{\fr}(G_K)$ of rank $d$ (with $K/\Qp$ finite extension), then it is shown in \cite[Rem. 3.3.5]{GL-OC} that $T/p^nT$ is torsion crystalline of weight $h(n) \leq n(p^{fd}+p-2)$ (where $f$ is the inertia degree of $K$). Namely, the growth of the (crystalline) weight of $T/p^nT$ is \emph{linear}.

During a conversation with Ruochuan Liu, he proposed the following question:
\begin{question} \label{question}
Let $T \in \Rep_{\Zp}^{\fr}(G_K)$. For each $n \ge 1$, suppose $T/p^nT$ is torsion semi-stable (resp. crystalline) of weight $h(n)$. Is it still possible to show that $T\otimes_{\Zp}\Qp$ is semi-stable (resp. crystalline), if we allow $h(n)$ to go to infinity?
\end{question}

By the paragraph above the question, it is necessary that $h(n)$ can not grow as fast as $n(p^{fd}+p-2)$ (in the case when $K/\Qp$ is a finite extension). So, one would expect that $h(n)$ has to grow more slowly than linear-growth. The first natural guess is the log-growth, and this is precisely what we obtained.

\begin{thm}  \label{thm intro}
Let $T \in \Rep_{\Zp}^{\fr}(G_K)$ of rank $d$. For each $n \ge 1$, suppose $T/p^nT$ is torsion semi-stable (resp. crystalline) of weight $h(n)$.
If $$h(n) < \frac{1}{2d}\log_{16} n, \forall n >>0,$$
 then $T\otimes_{\Zp}\Qp$ is semi-stable (resp. crystalline).
\end{thm}

One of the motivations of our work is the study of local-global compatibility problems in construction of Galois representations (associated to automorphic representations). Indeed, many such Galois representations are constructed via congruence methods. A good motivational explanation of the situation can be found in the introduction in Jorza's thesis \cite{MR2753218}. Namely, certain $p^n$-torsion semi-stable (or crystalline) representations will be constructed via congruence methods. However, the weights of these $p^n$-torsion representations grow (quite rapidly) to infinity, and so Theorem \ref{thm: Liu07} is no longer applicable. Unfortunately, our Theorem \ref{thm intro} seems also useless in this respect. To name one example, in the case \cite[Thm. 2.1, Thm. 3.1]{MR3039824}, the weights of these torsion representations grow exponentially. We do hope some of the techniques in our paper can be useful for future studies in local-global compatibility problems, perhaps combined with methods from analytic continuation of semi-stable periods.

\textbf{Notations}: Let $\O_{\overline K}$ be the ring of integers of $\overline K$. Let $R:=\varprojlim \limits_{x\to x^p} \O_{\overline K}/ p \O_{\overline K}$, and let $W(R)$ be the ring of Witt vectors of $R$. Let $A_{\t {cris}}$ be the usual period ring.

We fix a uniformizer  $\pi \in  \O_K$, and let $E(u) \in W(k)[u]$ be the Eisenstein polynomial of $\pi$. Define $\pi _n \in \overline K$ inductively such that $\pi_0 = \pi$ and $(\pi_{n+1})^p = \pi_n$. Then $\{\pi_n\}_{n \geq 0}$ defines an element $\upi \in R$, and let $[\underline \pi ]\in W(R)$ be the Techm\"uller representative of $\upi$.

Define $\mu _n \in \overline K$ inductively such that $\mu_1$ is a primitive $p$-th root of unity and $(\mu_{n+1})^p = \mu_n$.
Set $K_{\infty} : = \cup _{n = 1} ^{\infty} K(\pi_n)$, $K_{p^\infty}=  \cup _{n=1}^\infty
K(\mu_{n})$,
and $\hat K:=  \cup_{n = 1} ^{\infty} K(\pi_n, \mu_n).$
Let $G_{\infty}:= \gal (\overline K / K_{\infty})$, $H_\infty: =\gal(\overline K/\hat K)$, $H_{K}:=
\gal (\hat K/ K_\infty)$, and $\hat G: =\gal (\hat K/K) $.

When $V$ is a semi-stable representation of $G_K$, we let $D_{\textnormal{st}} (V):  = (B_{\st} \otimes_{\Qp} V^{\vee})^{G_K}$ where $V^{\vee}$ is the dual representation of $V$ (and $B_{\st}$ the usual period ring). The Hodge-Tate weights of $V$ are defined to be $i \in \mathbb Z$ such that $\mathrm{gr}^i ( K\otimes_{K_0} D_{\st}(V) ) \neq 0$.  For example, for the cyclotomic character $\varepsilon_p$, its Hodge-Tate weight is $\{ 1\}$.

\textbf{Acknowledgement}: I thank Ruochuan Liu for asking Question \ref{question}. I thank Andrei Jorza, Tong Liu for some useful discussions. I thank the anonymous referee(s) for useful comments which help to improve the exposition.
The paper is written when the author is a postdoc in University of Helsinki. The postdoc position is funded by Academy of Finland, through Kari Vilonen.

\section{Integral and torsion $p$-adic Hodge theory}
In this section, we recall some tools in integral and torsion $p$-adic Hodge theory.

\subsection{\'Etale $\varphi$-modules and \'etale $(\varphi, \tau)$-modules}
Recall that $\mathfrak{S}=W(k)[\![u]\!]$ with the Frobenius endomorphism $\varphi_{\huaS}\colon \huaS \to \huaS$ which acts on $W(k)$ via arithmetic Frobenius and sends $u$ to $u^p$.
Via the map $u\mapsto [\underline \pi]$, there is an embedding $\gs \hookrightarrow W(R)$ which is compatible with
Frobenious endomorphisms. Denote $\gs_n:=\gs/p^n\gs$.

Recall that $\O _\E$ is the $p$-adic completion of $\gs[1/u]$.
Our fixed embedding $\gs\hookrightarrow W(R)$ determined by  $\upi$
uniquely extends to a $\varphi$-equivariant embedding $\iota \colon \O_{\E}\hookrightarrow W(\t{Fr} R)$ (here $\t{Fr} R$ denotes the fractional field of $R$), and we identify $\O_{\E}$ with its image in $W(\t{Fr} R)$. Denote $\O_{\E, n}: =\O_{\E}/p^n \O_{\E}$.
We note that $\O_\E$ is a complete discrete valuation ring with uniformizer $p$ and residue field
$k (\!(\upi)\!)$ as a subfield of $\Fr R$. Let $\E$ denote the fractional field of $\O_\E$,  $\E ^\ur$ the maximal unramified extension of $\E$ inside $W(\t{Fr} R )[\frac 1 p ]$ and $\O_{\E ^\ur}$ the ring of integers of $\E ^\ur$. Set $\O_{\widehat{\E} ^\ur}$ the $p$-adic completion of $\O_{\E ^\ur}$.

\begin{defn}
Let $'\t{Mod}_{\O_\E}^\varphi$ denote the category of finite type $\O_\E$-modules $M $  equipped with a $\varphi_{\O_\E}$-semi-linear endomorphism $\varphi _M \colon M\to M$ such that $1 \otimes \varphi \colon \varphi ^*M \to M $ is an isomorphism. Morphisms in this category  are just $\O_\E$-linear maps compatible with $\varphi$'s. We call objects in $'\t{Mod}_{\O_\E}^\varphi$ {\em \'etale $\varphi$-modules}.
\end{defn}

Let $'\t{Rep}_{\Z_p}(G_{\infty}) $ (resp. $'\t{Rep}_{\Z_p}(G_{K}) $ ) denote the category of finite type $\Z_p$-modules $V$   with a continuous $\Z_p$-linear $G_{\infty}$ (resp. $G_K$)-action.
For $M $ in $'\t{Mod}_{\O_\E}^\varphi$, define
$$ V(M):= ( \O _{\widehat \E ^\ur} \otimes_{\O_\E} M) ^{\varphi =1}.  $$
For $V $ in  $'\t{Rep}_{\Z_p}(G_{\infty}) $, define
$$ \underline M(V):= ( \O _{\widehat \E ^\ur} \otimes_{\Z_p} V) ^{G_\infty}. $$

\begin{thm}[{\cite[Prop. A 1.2.6]{Fon90}}]
The functors $V$ and $\underline M$ induces an exact tensor equivalence between the categories $'\t{Mod}_{\O_\E} ^\varphi$  and $'\t{Rep}_{\Z_p} (G_{\infty})$.
\end{thm}

Recall that $H_\infty = \gal (\overline K / \hat K)$. Let $F_\tau : = (\t{Fr} R)^{H_\infty}$. As a subring of $W(\FrR)$,  $W(F_\tau)$ is stable on $G_K$-action and the action factors through $\hat G$.

\begin{defn} \label{defn phi tau mod}
An \'etale $(\varphi, \tau)$-module is a triple $(M, \varphi_M, \hat G)$ where
\begin{itemize}\item $(M , \varphi_M)$ is an \'etale $\varphi$-module;
\item $\hat G$ is a continuous $W(F_\tau)$-semi-linear $\hat G$-action on $\hat M : =W(F_\tau) \otimes_{\O_\E} M$, and $\hat G$ commutes with $\varphi_{\hat M}$ on $\hat M$, \emph{i.e.,} for
any $g \in \hat G$, $g \varphi_{\hat M} = \varphi_{\hat M} g$;
\item regarding $M$ as an $\O_\E $-submodule in $ \hat M $, then $M
\subset \hat M ^{H_{K}}$.
\end{itemize}
\end{defn}

Given an \'etale $(\varphi, \tau)$-module $\hat M = (M, \varphi_M, \hat G)$, we define
$$ \cT^* (\hat M):= ( W(\Fr R ) \otimes _{\O_\E} M) ^{\varphi=1} = \left ( W(\Fr R) \otimes_{W(F_\tau)} \hat M  \right )^{\varphi =1},$$
which is a representation of $G_K$.

\begin{prop}[{\cite[Prop. 2.1.7]{GL-OC}}] \label{prop-cT*}
Notations as the above. Then
\begin{enumerate} \item $\cT^*(\hat M)|_{G_\infty} \simeq V(M)$.
\item The functor $\cT^*$ induces an equivalence between the category of \'etale $(\varphi, \tau)$-modules and the category $'\t{Rep}_{\Z_p}(G_{K}) $.
\end{enumerate}
\end{prop}

\subsection{Kisin modules and $(\varphi, \hat G)$-modules}

\begin{defn} For a nonnegative integer $r$,
we write $'\sfi$  for the category of finite-type $\gs$-modules $\M$ equipped with
a $\varphi_{\gs}$-semilinear endomorphism $\varphi_\M \colon \M \to \M$ satisfying
\begin{itemize}
	\item the cokernel of the linearization $1\otimes \varphi \colon \varphi ^*\M \to \M$ is killed by $E(u)^r$;
	\item the natural map $\M \to \O_\E \otimes_{\gs} \M$ is injective.
\end{itemize}
 Morphisms in $'\sfi$ are $\varphi$-compatible $\gs$-module homomorphisms.
\end{defn}

We call objects in $'\sfi$ \emph{Kisin modules of $E(u)$-height $r$}. The category of {\em finite free Kisin modules of $E(u)$-height $r$},
denoted $\sfi$, is the full subcategory of $'\sfi$ consisting
of those objects which are finite free over $\gs$. We call an object $\M \in {'\sfi}$ a {\em torsion Kisin module of $E(u)$-height $r$} if $\M$ is killed by $p^n$ for some $n$. Since $E(u)$ is always fixed in this paper, we often drop $E(u)$ from the above notions.

Let $\M\in {'\sfi}$ be a Kisin module of height $r$, we define
 $$T^*_\gs (\M) := \left ( \M \otimes_{\gs}W(\FrR)\right)^{\varphi =1}.$$
Since $\gs \subset W(R) ^{G_\infty}$, we see that $G_\infty$ acts on $T^*_\gs (\M)$.
Note that this is the covariant version of the more usual (contra-variant) functor (see \cite[\S 2.3]{GL-OC}).

Now let us review the theory of $(\varphi, \hat G)$-modules.
We denote  by $S$  the $p$-adic completion of the divided power
envelope of $W(k)[u]$ with respect to the ideal generated by $E(u)$.
There is a unique map (Frobenius) $\varphi_S \colon S \to S$ which extends
the Frobenius on $\gs$. One can show that the embedding $W(k)[u] \to W(R)$ via $u \mapsto [\upi]$ extends to the embedding $S \inj A_\cris$. Inside $B^+_\cris =A_\cris[\frac 1 p]$, define a subring,
$$\mathcal{R}_{K_0}: =\left\{x = \sum_{i=0 }^\infty f_i t^{\{i\}}, f_i \in S[\frac{1}{p}] \text{ and } f_i \to 0 \text{ as } i \to +\infty \right\}, $$
where $t ^{\{i\}}= \frac{t^i}{p^{\tilde q(i)}\tilde q(i)!}$ and $\tilde q(i)$ satisfies $i = \tilde q(i)(p-1) +r(i)$ with $0 \leq r(i )< p-1$. Define $\hR := W(R)\cap \mathcal{R}_{K_0}$. One
can show that $\mathcal{R}_{K_0}$ and $\hR$ are stable under the $G_K$-action and the $G_K$-action factors through  $\hat G$ (see \cite[\S 2.2]{Liu10}). Let $I_+R$ be the maximal ideal of $R$ and $I_+ \hR = W(I_+R) \cap \hR$. By \cite[Lem. 2.2.1]{Liu10}, one has $\hR / I_+\hR\simeq \gs/ u \gs  = W(k)$.

\begin{defn}
Following \cite{Liu10}, a finite free (resp. torsion)$(\varphi, \hat
G)$-module of height $ r$ is a triple $(\M , \varphi, \hat G)$ where
\begin{enumerate}
\item $(\M, \varphi_\M)\in {'\sfi}$ is a finite free (resp. torsion) Kisin module of height $ r$;
\item $\hat G$ is a continuous $\hR$-semi-linear $\hat G$-action on $\hat \M: =\hR
\otimes_{\varphi, \gs} \M$;
\item $\hat G$ commutes with $\varphi_{\hM}$ on $\hM$, \emph{i.e.,} for
any $g \in \hat G$, $g \varphi_{\hM} = \varphi_{\hM} g$;
\item regard $\M$ as a $\varphi(\gs)$-submodule in $ \hM $, then $\M
\subset \hM ^{H_{K}}$;
\item $\hat G$ acts on $W(k)$-module $M:= \hM/I_+\hR\hM\simeq \M/u\M$ trivially.
\end{enumerate}
 Morphisms between  $(\varphi, \hat G)$-modules are morphisms of Kisin modules that commute with $\hat G$-action  on $\hM$'s.
\end{defn}

Given $\hM = (\M, \varphi_\M , \hat G) $ a $(\varphi, \hat G)$-module, either finite free or torsion, we define
$$ \hat T ^* (\hM) : = (W(\FrR) \otimes_{\varphi, \gs} \M ) ^{\varphi =1},$$
and it is a $\Z_p [G_K]$-module.

\begin{theorem}[{\cite[Thm 2.3.2]{GL-OC}}]  \label{thm-old-covariant}
\begin{enumerate}
\item  $\hat T^*$ induces an equivalence between the category of finite free
$(\varphi, \hat G)$-modules of height $r$ and the category of
$G_K$-stable $\Z_p$-lattices in semi-stable representations of $G_K$
with Hodge-Tate weights in $[-r, 0]$.
%\item $\hat T $ induces a natural $W(R)$-linear injection
%\begin{equation}\label{Eq: iota}
%\hat \iota: \  W(R)\otimes_{\varphi, \gs} \M  \longrightarrow \hat T^\v
%(\hat \M) \otimes_{\Z_p} W(R),
%\end{equation}
 %such that $\hat \iota$ is  compatible with Frobenius and $G_K$-actions on both sides. Moreover, $(\varphi(\gt))^r (\hat T^\v(\hM) \otimes_{\Z_p}W(R)) \subset\hat \iota( W(R)\otimes_{\varphi, \gs} \M).$
\item For $\hM$ a $(\varphi, \hat G)$-module, either finite free or torsion, there exists a natural isomorphism $T^*_\gs(\M) \ito \hat T^* (\hM)$ of $\Z_p[G_\infty]$-modules.
\end{enumerate}
\end{theorem}

We record a useful lemma which can identify crystalline representations from $(\varphi, \hat G)$-modules.
\begin{lemma} \label{Gao15 crys}
Suppose $K_{\infty} \cap K_{p^{\infty}}=K$ (which is always true when $p>2$), and let $\mhat$ be a finite free $(\varphi, \hat G)$-module.
Then $\hat T ^* (\hM)$ is a crystalline representation if and only if
$$(\wt{\tau}-1 )(\huaM) \in \mhat \cap (u^p\varphi(\mathfrak t) W(R)\otimes_{\varphi, \huaS}\huaM).$$
Here $\wt{\tau}$ is a topological generator of $G_{p^{\infty}}$ such that $\mu_{n}=\frac{\wt{\tau}(\pi_n)}{\pi_n}$ for all $n$, and $\mathfrak t \in W(R)\setminus pW(R)$ such that $\varphi(\huat)=\frac{pE(u)}{E(0)}\huat$ (note that $\huat$ is unique up to units of $\Zp$).
%$t= \lambda \varphi(\huat)$ where $\lambda = \Pi_{n=1}^{\infty} \varphi^n(\frac{c_0^{-1}E(u)}{p} )$ and $t=\log([\underline \varepsilon])$.
\end{lemma}
\begin{proof}
This is combination of \cite[Prop. 5.9]{GLS14} and \cite[Thm. 21]{Oze14}. Note that the running assumption $p>2$ in both papers is to guarantee $K_{\infty} \cap K_{p^{\infty}}=K$ (see the footnote in \cite[Prop. 4.7]{GLS14}). When $p=2$ and $K_{\infty} \cap K_{p^{\infty}}=K$, all the proofs still work.
\end{proof}

\begin{defn}
\begin{enumerate}
  \item Given an \'etale $\varphi$-module $M$ in  $'\t{Mod} _{\O_\E} ^{\varphi}$.   If $\M \in {'\sfi}$ is a Kisin module so that $M =  \O_\E \otimes_\gs \M $, then $\M$ is called a \emph{Kisin model} of $M$, or simply a \emph{model} of $M$.

  \item  Given $\hat M:  = (M, \varphi_M, \hat G_M)$ a torsion (resp. finite free) $(\varphi, \tau)$-module. A torsion (resp. finite free) $(\varphi, \hat G)$-module  $\hM : = (\M, \varphi_\M, \hat G)$ is called a \emph{model} of $\hat M$ if $\M$ is a model of $M$ and the isomorphism
$$ W(F_\tau) \otimes_{\hR} \hM \simeq W(F_\tau) \otimes_{\varphi, W(F_\tau)} \hat M$$
induced by $\O_\E \otimes_{\gs} \M \simeq M $ is compatible with $\hat G$-actions on both sides.
%Here $\hM : = \hR \otimes_{\varphi, \gs}\M$ and $\hat M: = W(F_\tau) \otimes_{\O_\E} M$.
\end{enumerate}
\end{defn}

Suppose $T$ is a $G_K$-stable $\Z_p$-lattice in a semi-stable representation  of $G_K$ with Hodge-Tate weights in $[-r, 0]$. Let $\hat M$ be the $(\varphi, \tau)$-module associated to $T$ via Proposition \ref{prop-cT*}, and let $\hM$ be the $(\varphi, \hat G)$-module associated to $T$ via Theorem \ref{thm-old-covariant}, then $\hM$ is a model of $\hat M$ (\cite[Lem. 2.4.3]{GL-OC}).

Now suppose that $T_n$ is a $p$-power torsion representation of $G_K$, and $\hat M_n$ the associated \'etale $(\varphi, \tau)$-module. Suppose there exists a surjective map of $G_K$-representations $f\colon L \onto T_n$ where $L$ is a semi-stable finite free $\Zp$-representation with Hodge-Tate weights in $[-r, 0]$ (we call such $f$ \emph{a loose semi-stable lift}). The loose semi-stable lift induces a surjective map (which we still denote by $f$) $f\colon \hat{\mathcal L} \onto \hat M_n$, where $\hat{\mathcal L}$ is the \'etale $(\varphi, \tau)$-module associated to $L$. Suppose $\hat{\mathfrak L}$ is the $(\varphi, \hat G)$-module associated to $L$, then it is easy to see that $f(\hat{\mathfrak L})$ is a $(\varphi, \hat G)$-model of $\hat M_n$.

\subsection{Torsion Kisin modules}
Let $\M \in {'\sfi}$ be a torsion Kisin module such that $M: = \M[\frac 1 u] $ is a finite free $\gs_n [\frac 1 u]$-module (i.e., the torsion $G_\infty$-representation associated to $M$ is finite free over $\mathbb Z/p^n \mathbb Z$).
For each $0 \le i < j \le n$, we define
$$ \M ^{i, j}:= \t{Ker}  (p ^i \M \overset{p ^{j-i}}{\longrightarrow} p ^ j \M ).$$
Following the discussion above \cite[Lem. 4.2.4]{Liu07}, we have $ \M ^{i, j} \in {'\sfi}$. We also have $\M^{i, j}[\frac 1 u]=p^{n-j+i}M$, and so it is finite free over $\mathcal O_{\mathcal E, j-i}$.

Define the function $\mfc(r):=4\cdot 4^re^2r^3$. This is (bigger than) the $\mfc$ in \cite[p. 653]{Liu07}.

The following three lemmas are extracted from \cite{Liu07}, and played important roles there.

\begin{lemma}\label{lem: mij lift}
Let $\hat{\mathfrak M}$ be a torsion $(\varphi, \hat G)$-module, and suppose it is torsion semi-stable, in the sense that it is the quotient of two finite free $(\varphi, \hat G)$-modules (with height $r$). Let $0 \leq i <j $, then $\hat{\mathfrak M}^{i, j}: = \t{Ker}  (p ^i \hat{\mathfrak M} \overset{p ^{j-i}}{\longrightarrow} p ^ j \hat{\mathfrak M} )$ is also torsion semi-stable. In fact, if $\hat{\mathfrak M} = \hat{\mathfrak L}/ \hat{\mathfrak L'}$, then there exists finite free $ \hat{\mathfrak N}$ and  $\hat{\mathfrak N'}$ such that $\hat{\mathfrak M}^{i, j}= \hat{\mathfrak N}/\hat{\mathfrak N'}$, which furthermore satisfy:
$$\hat T ^*(\hat{\mathfrak L})[\frac 1 p]  = \hat T ^*(\hat{\mathfrak L'})[\frac 1 p] =\hat T ^*(\hat{\mathfrak N})[\frac 1 p] =\hat T ^*(\hat{\mathfrak N'})[\frac 1 p]. $$
\end{lemma}
\begin{proof}
The lemma is extracted from the proof of \cite[Lem. 4.4.1]{Liu07}.

Let $\hat{ \mathfrak N} : =\Ker (p^j\hat{ \mathfrak L}  \to  p^j\hat{\M})$ and $\hat{ \mathfrak N'} : =\Ker (p^i\hat{ \mathfrak L}  \to  p^i\hat{\M})$. Both $\hat{ \mathfrak N}$ and $\hat{\mathfrak N'}$ are finite free $(\varphi, \hat G)$-modules by \cite[Cor. 2.3.8]{Liu07} (also note that the functor $\M \mapsto \hat{\mathcal R}\otimes_{\varphi, \gs}\M$ is exact, by \cite[Lem. 3.1.2]{CL11}).
There is a commutative diagram of  $(\varphi, \hat{G})$-modules:

 $$\xymatrix{ 0 \ar[r] & \hat{\mathfrak N'} \ar[d]  \ar[r] &p^i\hat{\mathfrak L}\ar[d]^{\simeq}  \ar[r]  & p^i\hat{\M} \ar[d]   \ar[r]& 0 \\
0 \ar[r] &  \hat{\mathfrak N} \ar[r] &p^j\hat{\mathfrak L} \ar[r]  & p^j\hat{\M} \ar[r] & 0  }$$
where all the vertical arrows are $\times p^{j-i}$ map.
By snake lemma, we have $$ 0 \to \hat{\mathfrak N'} \to \hat{\mathfrak N} \to \hat{\M}^{i, j} \to 0.  $$
\end{proof}

\begin{lemma}\label{lem1}
Suppose $\M, \N \in {'\sfi}$, both finite free over $\mfS_n$, and $\M[1/u]=\N[1/u]$. Suppose $n \geq \mfc(r)$, then $p^{\mfc(r)}\M = p^{\mfc(r)}\N  $.
\end{lemma}
\begin{proof}
This is \cite[Cor. 4.2.5]{Liu07}.
\end{proof}

\begin{lemma} \label{lem2}
Suppose $\M \in {'\sfi}$, such that $\M \otimes_{\gs} \mathcal O_{\mathcal E}$ is finite free over $\mathcal O_{\mathcal E, n}$. Suppose $n > 2\mfc(r)$, then $\M^{\mfc(r), n-\mfc(r)}$ is finite free over $\mfS_{n-2\mfc(r)}$.
\end{lemma}
\begin{proof}
This is extracted from \cite[Lem 4.3.1]{Liu07}.
\end{proof}

\section{Limit of torsion representations}
In this section, we prove our main theorem.

\begin{thm} \label{thm main}
Let $T \in \Rep_{\Zp}^{\fr}(G_K)$ of rank $d$. For each $n \ge 1$, suppose $T/p^nT$ is torsion semi-stable (resp. crystalline) of weight $h(n)$.
If $$h(n) < \frac{1}{2d} \log_{16} n, \forall n >>0,$$
 then $T\otimes_{\Zp}\Qp$ is semi-stable (resp. crystalline).
\end{thm}
\begin{proof}
Suppose $\hat M$ is the \'etale $(\varphi, \tau)$-module associated to $T$.
For each $n \geq 1$, since $T/p^nT$ is torsion semi-stable (resp. crystalline), let $\hat{\M}_n$ be a $(\varphi, \hat G)$-model of $\hat M/p^n \hat M$ associated to a loose semi-stable (resp. crystalline) lift of $T/p^nT$.

Denote $y_n =n^2 - 2\mfc(h(n^2)).$
It is easy to check that there exists some $n_0$ such that when $n \geq n_0$, we have:
\begin{itemize}
\item $ \mfc(h(n)) < \sqrt{n}$, which implies that $y_n >0$ and $y_{n+1} -y_n >0$;
\item and $y_n-n  > \mfc(\lceil \log_{16} (n+1) \rceil)$, where $\lceil \cdot \rceil$ is the ceiling function.
%$$ \mfc(h(n)) < \sqrt{n}, \quad y_{n+1} -y_n >0, \quad  y_n-n  > \mfc(\lceil \log_{16} (n+1) \rceil), $$ where $\lceil \cdot \rceil$ is the ceiling function.
\end{itemize}
Now, for any $n \geq n_0$, let
$$A_n:= \M_{n^2}^{ \mfc(h(n^2)) ,\quad n^2- \mfc(h(n^2)) } .$$
By Lemma \ref{lem2} (note that $A_n$ is a Kisin model of $M/p^{y_n}M$), $A_n$ is a Kisin module finite free over $\mfS_{y_n}$ (of rank $d$) of height bounded by $h(n^2)$.
Let $\M_n':= p^{ y_n-n}A_n$, then it is finite free over $\mfS_n$ of height bounded by $h(n^2)$.
Now we claim that $p\M_{n+1}'=\M_n', \forall n \geq n_0$.

To show the claim, consider $A_n$ and $p^{y_{n+1} -y_n}A_{n+1}$, both are finite free over $\mfS_{y_n}$ (and both are models of $M/p^{y_n}M$), with heights bounded by  $\log_{16} (n+1) $ (because $\max\{h(n^2), h((n+1)^2) \}  < \log_{16} (n+1) $). 
So by Lemma \ref{lem1}, we have
$$p^{\mfc(\lceil \log_{16} (n+1) \rceil)  } A_n   = p^{\mfc(\lceil \log_{16} (n+1) \rceil)  }p^{y_{n+1} -y_n}A_{n+1} .$$
Multiply both sides with $ p^{ y_n-n  - \mfc(\lceil \log_{16} (n+1) \rceil)}$, we get $p\M_{n+1}'=\M_n'$.

Now, define $\wt{\mathfrak M}: =\varprojlim_{n \geq n_0} \M_n'$, then it is a finite free $\gs$-module of rank $d$, and there is a natural $\varphi$-action on it. We claim that
\begin{itemize}
  \item $\wt{\mathfrak M}$ is a Kisin module, i.e., it is of finite $E(u)$-height.
\end{itemize}
To prove the claim, pick any $\gs$-basis of $\wt{\mathfrak M}$, and consider the matrix $A$ of $\varphi$ with respect to the basis. It is sufficient to show that there exists $b \in \gs$ such that $(\det A)\cdot b =E(u)^s$ for some $s$. Note that for each $n \geq n_0$, there exists some $B_n \in \Mat_{d}(\gs)$ such that $AB_n =E(u)^{h(n^2)}\textnormal{Id}  \pmod{p^n}$ (where $\textnormal{Id} $ is the identity matrix), and so
$$\det A \cdot \det B_n = E(u)^{dh(n^2)} \pmod{p^n}.$$
We then conclude by Lemma \ref{lem: E(u)} below.

Next we show that we can upgrade $\wt{\mathfrak M}$ to a $(\varphi, \hat G)$-module. The strategy is the quite similar to what is done in \cite[\S 5, 6, 7, 8]{Liu07}. However, because of the work \cite{Liu10} (which substantially used the results in \cite[\S 5, 6, 7, 8]{Liu07}), it is much easier now.

As we have shown that $\wt{\mathfrak M}$ is a Kisin module, it is obvious that ${T}_{\gs}^{*}({\wt{\mathfrak M}}) = T|_{G_{\infty}} $, and so $\wt{\mathfrak M}$ is a Kisin model of $M$. Consider the $\hat G$-action on
$$W(F_{\tau})\otimes_{\varphi, W(F_{\tau})}\hat M=W(F_{\tau})\otimes_{\varphi, \gs} \wt{\mathfrak M}.$$
By Lemma \ref{lem: mij lift}, all the modules
$$\hat{\M}_n' =\t{Ker}  (p ^{\mfc(h(n^2))+y_n-n} \hat{\M}_{n^2} \overset{p ^{n}}{\longrightarrow} p^{n^2- \mfc(h(n^2))} \hat{\M}_{n^2} )$$
are also torsion semi-stable, and so the $\hat G$-actions on $W(F_{\tau})\otimes_{\varphi, \gs} {\mathfrak{M}_n'}$ descends to $\hat G$-actions on $\hat{\mathcal R} \otimes_{\varphi, \gs}  {\mathfrak{M}_n'}$. By taking inverse limit, the $\hat G$-action on $W(F_{\tau})\otimes_{\varphi, \gs} \wt{\mathfrak M}$ descends to a $\hat G$-action on $\hat{\mathcal R} \otimes_{\varphi, \gs} \wt{\mathfrak M}$, and so $(\wt{\mathfrak M}, \varphi, \hat G)$ is a $(\varphi, \hat G)$-module.
Now it is obvious that $\hat{T}^{*}(\hat{\wt{\mathfrak M}}) = T $, and so $T$ is semi-stable.

Now we only need to deal with the crystalline case. When the conditions in Lemma \ref{Gao15 crys} is satisfied (that is, when $p>2$, or when $p=2$ and $K_{\infty} \cap K_{p^{\infty}}=K$), then the $\hat G$-actions on $\hat{\mathcal R} \otimes_{\varphi, \gs}  {\mathfrak{M}_n'}$ satisfy the (torsion version of the) conclusion in \textit{loc. cit.}, and so the $\hat G$-action on $\hat{\wt{\mathfrak M}}$ satisfies the conclusion in \textit{loc. cit.} as well (note that $u^p\varphi(\mathfrak t)W(R)$ is $p$-adically closed in $W(R)$), and so $\hat{\wt{\mathfrak M}}$ is crystalline.

When $p=2$ and $K_{\infty} \cap K_{p^{\infty}} \neq K$, then we can argue similarly as in the very final paragraph of \cite{Liu10} (which is the errata for \cite{Liu07}). Namely, we can show that $T$ is crystalline over both $K(\pi_1)$ and $K(\mu_4)$, and so $T$ is crystalline over $K(\pi_1)\cap K(\mu_4)=K$.
\end{proof}

\begin{lemma} \label{lem: E(u)}
Let $\alpha \in \gs$, suppose there exists some $n_0$ such that for any $n \geq n_0$, there exists $\beta_n \in \gs$ such that
$$\alpha \beta_n = E(u)^{dh(n^2)}  \pmod{p^n \gs},$$
where $d$ and $h(n^2)$ are as in Theorem \ref{thm main}. Then there exist some $s \in \mathbb Z^{\geq 0}$ and $\gamma \in \gs$ such that $\alpha \gamma =E(u)^s$.
\end{lemma}

Before we prove the lemma, we recall a useful lemma. Note that $E(u)$ is not a zero divisor in $\gs_n, \forall n \geq 0$, so it is OK to do ``division by $E(u)$" in $\gs_n$.
\begin{lemma}[{\cite[Lem. 4.2.2]{Liu07}}]
Suppose $f, g \in \gs_n$ with $n\geq 2$, suppose $E(u)| fg \pmod{p^n}$. Then we have
 $$E(u)| f \pmod p^{\lfloor \frac n 2 \rfloor} \textnormal{ or } E(u)| g \pmod{p^{\lfloor \frac n 2 \rfloor}}, $$
where $\lfloor \cdot \rfloor$ is the floor function.
\end{lemma}
The following easy corollary is convenient for our use.
\begin{cor} \label{cor conv}
Suppose $f, g \in \gs_{2^n}$. Suppose $E(u)^k | fg  \pmod{p^{2^n}}$ where $k<n$. Then we will have
 $$E(u)^a | f \pmod{p^{2^{n-k}}}, \textnormal{ and } E(u)^b| g \pmod{p^{2^{n-k}}}$$
 for some $a, b \geq 0$ such that $a+b=k$.
\end{cor}

\begin{proof}[Proof of Lemma \ref{lem: E(u)}]
First we have $u \nmid \alpha$. This is because when $n$ is big enough, $f_0(\alpha \beta_n)= f_0(E(u)^{dh(n^2)} +p^n \theta_n )$ for some $\theta_n \in \gs$, and the right hand side is $\neq 0$, because $dh(n^2) <n$. Here $f_0$ is the $W(k)$-linear map $\gs \to W(k)$ with $f_0(u)=0$.

Next, suppose $E(u)^{x_n} |\alpha $ in $\gs_n$. Then we claim that there exists $s$ such that $x_n \leq s, \forall n \geq n_0$.
To prove the claim, write $\alpha=E(u)^{x_n}\theta_{1, n} + p^n \theta_{2, n}$ for some $\theta_{1, n}, \theta_{2, n} \in \gs, \forall n \geq n_0$. Since $f_0(\alpha) \neq 0$ and $p | f_0(E(u)) $, it is easy to see that the sequence $\{x_n\}$ has to be bounded.

Finally, we claim that for all $n >> 0$, there exists $\gamma_n \in \gs_n$ such that $\alpha \gamma_n=E(u)^s \pmod{p^n}$ (note that such $\gamma_n$ is unique). We only need to show existence of such $\gamma_n$ for a sequence $\{n_m\}$ going to infinity.

For all $m> \textnormal{max} \{s, n_0\}$, consider $n=16^{dm}$, so $h(n^2) \leq m$.
We can and do assume that
$\alpha \beta_n =E(u)^{dm} \pmod{p^{16^{dm}}}$ (when $h(n^2)<m$, we can simply multiply some $E(u)$-power to $\beta_n$, and it does not affect our result). We want to show that there exists $\gamma_n \in \gs_n$ such that $\alpha \gamma_n=E(u)^s \pmod{p^{16^{dm}}}$.

Take any $m'>2m$, and let $n'=16^{dm'}$, so we have $\alpha \beta_{n'} =E(u)^{dm'} \pmod{p^{16^{dm'}}}$.
Apply Corollary \ref{cor conv}, then we will have
$$E(u)^a |  \alpha \pmod{p^{2^{4dm'-dm'}}} \textnormal{ and }  E(u)^b | \beta_{n'} \pmod{p^{2^{4dm'-dm'}}}$$
where $a+b=dm'$. However, we always have $a \leq s$, and so $b \geq dm'-s$ (and $dm'-s>0$ because $m>s$).
That is, we now have (note that $2^{4dm'-dm'} >16^{dm}$ ),
$$\alpha \beta_{n'} =E(u)^{dm'} \pmod{p^{16^{dm}}} \textnormal{ and } E(u)^{dm'-s} \mid \beta_{n'}  \pmod{p^{16^{dm}}}.$$ 
So we can simply let $\gamma_n = \frac{\beta_{n'}}{E(u)^{dm'-s} } $.

Now simply let $\gamma:=\varprojlim_{n >>0} \gamma_n$, and we are done.

\end{proof}

\bibliographystyle{alpha}

\end{document}